\documentclass[12pt]{amsart} 
\usepackage{amsfonts,amsmath,amsthm,oldgerm,amssymb,amscd,} 
\usepackage{setspace} 
\usepackage[all,cmtip]{xy} 

\usepackage{fancyhdr} 
\usepackage{psfrag} 
\usepackage{graphicx} 
 
\usepackage[indent]{caption} 
\usepackage{indentfirst} 
\usepackage{float} 
\usepackage{latexsym} 
\usepackage{graphics} 
\usepackage{verbatim}

\newcommand{\R}{{\mathbb R}} 
\newcommand{\Z}{{\mathbb Z}} 
\newcommand{\C}{{\mathbb C}}

\renewcommand{\H}{{\mathbb H}} 

\newtheorem{theorem}{Theorem}
\newtheorem{lemma}{Lemma}
\newtheorem{corollary}{Corollary}
 
\theoremstyle{definition}

\newtheorem{definition}{Definition} 
\numberwithin{definition}{section} 
 
\newtheoremstyle{mystyle}{}{}{}{}{\bfseries}{.}{ }{} 
\theoremstyle{mystyle} 
\newtheorem{remark}{Remark}

\begin{document} 
 
\title[]{On a 
strong  multiplicity one property for \\   
the length spectra of even dimensional  \\  compact hyperbolic spaces} 
 
\date{} 
 
\author{Chandrasheel Bhagwat} 
\author{C.S.Rajan} 
 
\address{Tata Institute of Fundamental Research\\ 
Homi Bhabha Road\\ 
Mumbai 400005, India.} 
 
\email{chandra@math.tifr.res.in, rajan@math.tifr.res.in} 
 
\begin{abstract} 
We prove a strong multiplicity one theorem for the length spectrum of compact  
even dimensional hyperbolic spaces i.e. if all but finitely many
closed geodesics for two compact even
dimensional hyperbolic spaces 
 have the same length, 
 then all closed geodesics have the same length. 
\end{abstract} 
 
\maketitle 
\section{Introduction} The analogy between the spectrum and arithmetic 
of hyperbolic surfaces was studied by A. Selberg. The primitive closed 
geodesics on a hyperbolic surface with finite volume can be considered 
as analogues of prime numbers and Selberg established an analogue of 
the prime number theorem for primitive closed geodesics (see 
\cite{He}). The concept of the length spectrum can be introduced, and 
the Selberg trace formula establishes a relationship between the 
spectrum of the Laplacian acting on functions on a compact hyperbolic 
surface and the length spectrum of the surface. This can be generalized to higher dimensional compact hyperbolic 
spaces $X_{\Gamma}$ which are quotients of the hyperbolic $n$-space 
$\H_{n}$ by torsion free uniform lattices $\Gamma$ acting on $\H_{n}$ 
by isometries. 
 
Define the length spectrum of $X_{\Gamma}$ to be the function $$ L_{\Gamma} : \R 
\rightarrow \mathbb{N} \cup \left\{ 0 \right\}$$ which to a real 
number $l$, assigns the number of closed geodesics of length $l$ in 
$X_{\Gamma}$.  Two compact hyperbolic spaces $X_{\Gamma_{1}}$ and 
$X_{\Gamma_{2}}$ are said to be  \textit{length-isospectral} if 
$L_{\Gamma_{1}}(l) = L_{\Gamma_{2}}(l)$ for all real numbers $l$.  
 
In this article we establish  
the following strong multiplicity one type property for 
the length spectrum of even dimensional compact hyperbolic spaces: 
 
\begin{theorem}\label{liso} Let $\Gamma_1$ and $\Gamma_2$ be uniform 
lattices in the isometry group of $2n$-dimensional hyperbolic 
space. Suppose  $L_{\Gamma_{1}}(l)$ = $L_{\Gamma_{2}}(l)$ for all but 
finitely many real numbers $l$. Then $L_{\Gamma_{1}}(l)$ = 
$L_{\Gamma_{2}}(l)$ for  all real numbers $l$ i.e., the corresponding 
compact hyperbolic spaces  are length-isospectral. 
\end{theorem} 
 
\remark This result can be considered in analogy with the 
classical strong multiplicity one theorem for cusp 
forms. Suppose $f$ and $g$ are newforms for some Hecke congruence 
subgroup $\Gamma_{0}(N)$ such that the eigenvalues of the Hecke 
operator at a prime $p$ are equal for all but finitely many primes 
$p$. Then the strong multiplicity one theorem of Atkin and Lehner 
states that $f$ and $g$ are equal (\textit{cf}. ~\cite[p.125]{La}). 
 
The proof of Theorem \ref{liso} 
uses the analytic properties of Ruelle zeta 
function, in particular the functional equation satisfied by it. The 
method is similar to the proof of the strong multiplicity one theorem  
for $L$-functions of the Selberg class given in 
\cite{MM}. For odd dimensions, the form of the functional equation 
satisfied by the Ruelle zeta function is different, and hence our 
method does not shed any light when the dimension is odd.

\remark Using an analytic version of the Selberg Trace formula, 
J. Elstrodt, F. Grunewald, and J. Mennicke  
proved a different version of Theorem \ref{liso}  
for $n = 2, ~3$.~(\cite[Theorem 3.3, p.203]{EGM}).  
Their notion of length spectrum is different from the  
notion defined in this paper. 
 
\section{Preliminaries} 
 
Let $G$ be the connected component of the isometry group 
$\text{SO}(n,1)$ of $\mathbb{H}_{n}$. Fix a maximal compact subgroup 
$K$ of $G$. Hence $G/K$ is homeomorphic to the $n$-dimensional 
hyperbolic space $\H_{n}$. Let $X_{\Gamma}$ be a compact 
$n$-dimensional hyperbolic space of the form 
\begin{equation} X_{\Gamma} = \Gamma \backslash G/K, 
\end{equation} 
where $\Gamma$ is a torsion free uniform lattice in $G$. 
 
Let $C$ be a free homotopy class in $X_{\Gamma}$. Let $\gamma_C$ be  a
curve in $C$ defined on the interval $\left[0,a \right]$, for some
positive real number $a$.  There  exists a lift $\tilde{\gamma}_C$ of
$\gamma_C$ to $\H_{n}$. Since  $\Gamma$ acts on $\H_{n}$ by deck
transformations, the action is  transitive on fibres and hence there
exists $g_C \in \Gamma$ such that  $\tilde{\gamma}_C(a)=
g_{C}.\tilde{\gamma}(0)$. The element $g_C$ is determined upto
conjugacy in $\Gamma$. For an element $g$ of $\Gamma$, let $[g]$
denote the conjugacy class   of $g$ in $\Gamma$. It can be seen that
the map $C  \mapsto \left[ g_C \right]$ is a  bijection between the
collection of free homotopy classes in $X_{\Gamma}$ and  the set of
conjugacy classes of elements  in $\Gamma$.  
 
It is known that in any  
negatively curved compact Riemannian manifold every free 
homotopy class $C$ contains a unique closed geodesic which we continue 
to denote by $\gamma_C$. Since  
 $X_{\Gamma}$ has constant negative 
sectional curvature $-1$, we have the following lemma : 
  
\begin{lemma}\label{freehom}  
There is a bijective correspondence 
between the set of closed geodesic classes in $X_{\Gamma}$ and the set 
of conjugacy classes of $\Gamma$ given by, 
 \[ \left[\gamma_C \right] \rightarrow \left[ g_C \right]. \] 
\end{lemma} 
 
\begin{definition} Let $\gamma \in \Gamma$. The length $\ell(\gamma)$ 
of the conjugacy class $[\gamma]$ is defined to be the length of the 
unique closed geodesic in the free homotopy class corresponding to 
$\gamma$ in $X_{\Gamma}$. 
\end{definition} 
 
\remark  Let $G=KAN$ be an Iwasawa decomposition for $G=SO(n,1)$,
where $A$ is the connected component of the split part of a Cartan
subgroup of $G$.  Let $\mathfrak{a}$ be it's Lie algebra 
Fix an element $X_0$ of norm 1 in $\mathfrak{a}$ with respect 
to the Killing form. 
 Let $M$ denote the centralizer 
of $A$ in $K$. It is known that every semisimple element $\gamma$ of $G$  
is conjugate to an element $m_{\gamma} a_{\gamma}$ in $G$ where  
$m_{\gamma} \in M$ and $a_{\gamma} \in A$ .  If $\gamma$ is in $\Gamma$ and the length of $[\gamma]$ is 
$\ell(\gamma)$, then $a_{\gamma} = \exp(\ell(\gamma)X_{0}$.

 Since $\Gamma$ is an uniform lattice in $G$, it is known that any
 relatively compact subset of $G$ intersects only finitely many
$G$-conjugacy classes of elements in $\Gamma$. Hence 
it follows that there are only 
finitely many closed geodesics of a fixed length. 
 
We define the  length spectrum of $X_{\Gamma}$ to be  the function 
$L_{\Gamma}$ defined on $\R$ by, 
$$L_{\Gamma}(l) = \text{The number of conjugacy classes  
$[\gamma]$ in}~ \Gamma~\text{such that}~\ell(\gamma) = l.$$ 
  
\section{Ruelle Zeta function}\label{rue} 
In this section, we recall the  definition and some   
properties of the Ruelle Zeta function attached to  
a compact hyperbolic space.  
 
\begin{definition} 
A conjugacy class $[\gamma]$ of $\Gamma$ is called primitive if $[\gamma] 
\neq [\delta^n]$ for any integer $n \geq 2$ and $\delta \in \Gamma$. 
\end{definition} 
 
Define the  primitive length spectrum of $X_{\Gamma}$ to be  the 
function $PL_{\Gamma}(l)$ defined by, 
$$PL_{\Gamma}(l) = \text{The number of primitive conjugacy classes  
$[\gamma]$ in}~ \Gamma~\text{such that}~\ell(\gamma) = l.$$ 
 
By  $P_{\Gamma}$, we denote the set of primitive conjugacy 
classes in $\Gamma$. For $\gamma \in  \Gamma$,  define $N(\gamma) = 
{e}^{\ell~\left(\gamma\right)}.$ The Ruelle Zeta function is defined 
by the infinite product: 
\begin{equation}\label{ruelle} R_{\Gamma}(z)  =  
\prod \limits_{\gamma 
\in P_\Gamma}(1-(N(\gamma))^{-z})^{(-1)^{(n-1)}} 
\end{equation} 
  
It can be shown that the above product converges uniformly on compact 
sets in the  right half plane $Re(z) > \rho$ (for some $\rho$ which 
depends on $n$). Hence it defines a holomorphic function on  the 
domain $Re(z) > \rho$. The following theorem describes the analytic 
properties of the Ruelle zeta function of relevance to us (see 
\cite[page 126, Theorems 4.3 and 4.4]{BO}):  
 
\begin{theorem} 
The Ruelle Zeta function $R_{\Gamma}$ admits a meromorphic continuation 
to the whole complex plane $\mathbb{C}$ and  
 satisfies the following functional equation : 
\begin{itemize} 
\item $X_{\Gamma}$ is even dimensional of dimension $2n$. 
 \begin{equation}\label{FE} 
R_{\Gamma}(z)R_{\Gamma}(-z)= {\displaystyle\biggl[\sin\left({\frac{\pi z}{ T}}\right)\biggr]}^{2n.\chi\left(X_{\Gamma}\right)}  
\end{equation} 
where $T$ is a positive constant independent of $\Gamma$ and $\chi(X_{\Gamma})$ 
is the Euler Characteristic of $X_{\Gamma}$.\\ 
 
\item $X_{\Gamma}$ is odd dimensional of dimension $2n+1$.  
\begin{equation}\label{FE2} 
\frac{R_{\Gamma}(z)}{R_{\Gamma}(-z)} =  
\exp{\displaystyle\biggl[{\frac{\left(4\pi(n+1)~  
\text{vol}(X_{\Gamma})\right)z}{TC}}\biggr]}  
\end{equation} 
where $T,~C$ are positive constants independent of $\Gamma$. 
\end{itemize} 
\end{theorem} 
 
\section{Proof of Theorem \ref{liso}} 
 
We begin with the following preliminary lemma. 
\begin{lemma}\label{lem} 
Let $e \neq \alpha \in \Gamma$. Then the centralizer $C_{\Gamma}(\alpha)$ 
is cyclic.  
\end{lemma} 
 
\begin{proof}   Let $h$ be an element of $C_{\Gamma}(\alpha)$. Then
$h$ is conjugate to an element  $g = m_{h} a_{h}$ in $G$ where $m_{h}
\in M$ and $a_{h} \in A$. Let $H = \left\lbrace  a_{h} : h \in
C_{\Gamma}(\alpha) \right\rbrace $. Since elements of $M$ commute with
elements of $A$,  it follows that $H$ is a subgroup of $A$. Since
$\Gamma$ is discrete, so is $H$ and consequently $H$ is cyclic as $A$
is isomorphic to $\R$. Let $a_{h_{0}}$ be a generator of $H$.  We show
that $C_{\Gamma}(\alpha)$ is generated by $h_{0}$. 
 
Given $h \in C_{\Gamma}(\alpha)$, $a_{h} = a_{h_{0}}^k$ for some $k
\in \Z$. Then   $h.h_{0}^{-k}$ is conjugate to some element of
$M$. Since $\Gamma$ is discrete and torsion-free,   and $M$ is
compact,  it follows that
$h.h_{0}^{-k} = e$. Thus $C_{\Gamma}(\alpha)$ is cyclic.  
\end{proof} 
 
\begin{corollary}\label{cor} 
Let $[\alpha]$ and $[\beta]$ be primitive conjugacy classes such that $\left[\alpha^k\right] = 
\left[\beta^{k'}\right]$ for some natural numbers $k, ~k'$. 
 Then $[\alpha] = [\beta]$ and $k = k'$. 
\end{corollary} 
 
\begin{proof} Let $[g] = \left[\alpha^k\right] = 
\left[\beta^{k'}\right] $. By Lemma \ref{lem} it follows that $C_{\Gamma}(g)$ is cyclic. Since both  
$\alpha$ and $\beta$ are primitive conjugacy classes it follows that $[\alpha] = [\beta]$ and $k = k'$. 
\end{proof} 
 
 Now we prove the following lemma,  
recovering the length spectrum from the primitive length spectrum :  
 \begin{lemma}\label{eq} 
 $$L_{\Gamma}(l) = \sum \limits_{d ~=~ 1}^{\infty} PL_{\Gamma}\left(\frac{l}{d}\right) \quad \forall~l \in \R.$$ 
 \end{lemma} 
  
  \begin{proof} Let $d \geq 0$ be an integer and $l$ be a real 
number. Note that if $\left[\gamma\right]$ is a primitive conjugacy 
class of length $l/d$, then   $\left[\gamma^{d}\right]$ is a conjugacy 
class of length $l$.  Hence by Corollary \ref{cor}, 
  \[ L_{\Gamma}(l) \geq \sum \limits_{d ~=~ 1}^{\infty} 
PL_{\Gamma}\left(\frac{l}{d}\right) \quad \forall~l \in \R.\] 
  \noindent Conversely, given a conjugacy class $[\gamma]$ of length 
$l$, the associated primitive class is of length $l/d$ for some 
non-negative integer $d$. Hence the other inequality follows and this 
proves the lemma.  
  \end{proof} 
   
\begin{corollary}\label{prelim} Let ${\Gamma_{1}}$ and ${\Gamma_{2}}$ 
be torsion free uniform lattices in SO$(n,1)$.  Let $X_{\Gamma_{1}}$ 
and $X_{\Gamma_{2}}$ be the associated compact hyperbolic spaces. 
Suppose  
\[ PL_{\Gamma_{1}}(l) = PL_{\Gamma_{2}}(l) \quad \forall~l \in \R.\]  
Then  
\[ L_{\Gamma_{1}}(l) = L_{\Gamma_{2}}(l) \quad \forall~l \in \R.\]  
i.e. the spaces $X_{\Gamma_{1}}$ and $X_{\Gamma_{2}}$ are 
length isospectral. 
  \end{corollary}

In order to prove Theorem \ref{liso}, by Corollary \ref{prelim}, it is 
enough to show that the primitive  length spectra of $X_{\Gamma_1}$ and $X_{\Gamma_2}$ 
are equal. Let $R_{\Gamma_{1}}(z)$ and $R_{\Gamma_{2}}(z)$ be the 
respective Ruelle Zeta functions.  Let dim $X_{\Gamma_{1}} = 
\text{dim}~ X_{\Gamma_{2}} = 2n$. From the definition of the Ruelle Zeta 
function, we have in the region  $\mbox{Re}(z) > \rho$,  
 
\begin{equation}\label{ratioeqn} 
\frac{R_{\Gamma_{1}}(z)}{R_{\Gamma_{2}}(z)} \quad = \quad \frac{\prod 
\limits_{\gamma \in P_{\Gamma_1}} {\displaystyle\biggl(1 - 
N(\gamma)^{-z}\biggr)}^{(-1)^{(2n-1)}}}{\prod \limits_{\gamma' \in 
P_{\Gamma_2}} 
{\displaystyle\biggl(1-N(\gamma')^{-z}\biggr)^{(-1)^{(2n-1)}}}} 
\end{equation}  
 
Under the hypothesis of Theorem (\ref{liso}), all but 
finitely many factors cancel out. Hence there exist finite sets $S_1$ 
and $S_2$ such that for $\mbox{Re}(z) > \rho$,  
 
\[ 
 \frac{R_{\Gamma_{1}}(z)}{R_{\Gamma_{2}}(z)}\quad = \quad \frac{\prod \limits_{j \in S_2} \displaystyle\biggl{(1 - N(\gamma'_j)^{-z}\biggr)}} 
{\prod \limits_{i \in S_1}{\displaystyle\biggl{(1 - N(\gamma_i)^{-z}\biggr)}}}. 
\]  
 
Since the products in the above equation are over finite sets, the 
ratio $R_{\Gamma_{1}}(z)/R_{\Gamma_{2}}(z)$  defines a meromorphic 
function on the whole complex plane. Since 
$R_{\Gamma_{1}}/R_{\Gamma_{2}}$ admits a meromorphic continuation, the 
two expressions must agree for all $z \in \C$, i.e. 
\[ \frac{R_{\Gamma_{1}}(-z)}{R_{\Gamma_{2}}(-z)}\quad  = \quad \frac{\prod \limits_{j \in S_2} 
\displaystyle\biggl{(1-N(\gamma'_j)^{z}\biggr)}}{\prod \limits_{i \in 
S_1} \displaystyle\biggl{(1- N(\gamma_i)^{z}\biggr)}} 
\]

From the functional equation (\ref{FE}) applied to $X_{\Gamma_{1}}$ and $X_{\Gamma_{2}}$,  
 
\begin{equation} \label{lhs} 
\begin{split} 
\frac{R_{\Gamma_{1}}(z) R_{\Gamma_{1}}(-z)}{R_{\Gamma_{2}}(z) 
R_{\Gamma_{2}}(-z)} \quad & =\quad  
{\sin{(\pi z / T)}}^{2n.(\chi(X_{\Gamma_{1}})\ -\ \chi(X_{\Gamma_{2}}))} 
\quad \\ 
& = \quad \frac{\prod \limits_{j \in S_2}  
\displaystyle\biggl{(1-N(\gamma'_j)^{-z}\biggr)}\ \displaystyle\biggl{(1-N(\gamma'_j)^{z}\biggr)}} 
{\prod \limits_{i \in S_1}\displaystyle\biggl{(1 - N(\gamma_i)^{-z}\biggr)}\ 
 \displaystyle\biggl{(1- N(\gamma_i)^{z}\biggr)}}  
\end{split} 
\end{equation} 
  
If $\chi(X_{\Gamma_{1}}) - \chi(X_{\Gamma_{2}})\ \neq\ 0$, then every 
integer multiple of $T$ is either a zero or a pole of the left hand 
side in equation (\ref{lhs}). On the other hand, all zeros and poles 
of right hand side are on imaginary axis. This leads to a 
contradiction. Hence we conclude that $\chi(X_{\Gamma_{1}}) = \chi(X_{\Gamma_{2}})$ and the 
left hand side in (\ref{lhs}) is identically $1$ i.e. for all $z \in 
\C$, 
  
\[\prod \limits_{i \in S_1} (1 - N(\gamma_i)^{-z}) 
 \prod \limits_{i \in S_1} (1 - N(\gamma_i)^{z}) \ 
= \prod \limits_{j \in S_2} (1 - N(\gamma'_j)^{-z}) 
 \prod \limits_{j \in S_2} (1 - N(\gamma'_j)^{z}) 
\]  
The function on the left side of the above equation vanishes at $Z$ 
if and only if  $z = 2\pi i k / \ell(\gamma_{j})$ for some $k \in 
\mathbb{Z}$ and $j \in S_1$. Similarly  the function on the right side 
of the above equation vanishes at exactly when $z = 2\pi i k / 
\ell(\gamma'_{j})$ for some $k \in \mathbb{Z}$ and $j \in S_2$.  Hence 
we get an  equality of sets with multiplicity : 
\[ \left\{ \frac{2 \pi i k}{\ell(\gamma_{j})}  : k \in \mathbb{Z}; \ j 
\in S_1 \right\}\quad  = \quad \left\{ \frac{2\pi i k}{\ell(\gamma'_{j})} : k \in 
\mathbb{Z};\ j \in S_2 \right\}. \] 
 
Hence we conclude that all factors in equation 
(\ref{ratioeqn}) must cancel out. Consequently, $PL_{\Gamma_{1}}(l)= 
PL_{\Gamma_{2}}(l)$ for all $l \in \mathbb{R}$, and  
Theorem \ref{liso} follows from  Corollary \ref{prelim}.

\begin{remark} In the odd dimensional case,  arguing as above  
and using the functional equation (\ref{FE2}), yields the following 
equation: 
\begin{equation*}\label{oddcase}  
\exp{{\displaystyle\biggl[\frac{(4 
\pi (n+1) (\text{vol}(X_1) - \text{vol}(X_2))z)}{TC}\biggr]}}= 
\frac{\prod \limits_{i \in S_1}{\displaystyle\biggl(1 - 
N(\gamma_i)^{-z}\biggr)} \prod \limits_{j \in S_2}{\displaystyle\biggl 
(1-N(\gamma'_j)^{-z}\biggr)}}{\prod \limits_{i \in S_1} 
{\displaystyle\biggl (1- N(\gamma_i)^{z}\biggr)} \prod \limits_{j \in 
S_2} {\displaystyle\biggl (1-N(\gamma'_j)^{z}\biggr)}} 
\end{equation*}  
Since  
\[ \frac{1-a^{z}}{1-a^{-z}}=-a^z,\] 
the above equation  yields no information. Hence we 
cannot conclude anything in the odd dimensional case.  
\end{remark}

\end{document}